\documentclass[12pt]{article}%
\usepackage{amssymb}
\usepackage{amsfonts}
\usepackage{amsmath}
\usepackage{graphicx}
\usepackage{tikz}
\usepackage{color}%
\providecommand{\U}[1]{\protect\rule{.1in}{.1in}}
\newtheorem{theorem}{Theorem}[section]

\newtheorem{conjecture}[theorem]{Conjecture}

\newtheorem{Fact}[theorem]{Fact}

\begin{document}

\title{\textbf{A Counterexample to an Eternal Domination Conjecture}}

\author{Tom Adamczewski \\ Epoch AI\\London, UK\\{\small tmkadamcz@gmail.com} \and William F. Klostermeyer\\School of Computing\\University of North Florida\\Jacksonville, FL 32224-2669\\{\small wkloster@unf.edu}}
\date{}
\maketitle

\begin{abstract}
A graph $G$ with 243 vertices is shown having $\gamma(G) = \gamma^{\infty}(G) < \theta(G)$, disproving the Gamma-Theta conjecture for eternal dominating sets.
\end{abstract}

\color{red}%
\color{black}%

\section{Definitions}

Let $G=(V, E)$ be a simple, undirected graph, Denote the open and closed neighborhoods of a vertex $x \in V$ by
$N(x)$ and $N[x]$, respectively. A \emph{dominating set }of $G$ is a set $D\subseteq V$ with the
property that for each $u\in V-D$, there exists $x\in D\ $adjacent
to $u$. The minimum cardinality among all dominating sets is the
\emph{domination number}  $\gamma(G)$.  The maximum cardinality among all cliques is the \emph{clique number} $\omega(G)$.

The \emph{clique covering number }$\theta(G)$ is the minimum number $k$ of
sets in a partition $V_{1},...,V_{k}$ of $V$ such that the subgraph $G[V_{i}]$
induced by each $V_{i}$ is a clique. Observe that $\theta(G)$ equals the
chromatic number $\chi(\overline{G})$ of the complement $\overline{G}$ of $G$.
Since $\chi(G)=\omega(G)$ if $G$ is perfect, and $G$ is perfect if and only if
$\overline{G}$ is perfect, it follows that $\alpha (G)=\theta(G)$ for all perfect graphs.

\section{Eternal Domination}

Let $\mathcal{D}_{k}$ be the collection of all dominating sets of $G$ of fixed
cardinality $k$. For a set $D\in\mathcal{D}_{k}$, there is a single guard
located on each vertex of $D$ and therefore we think of $D$ as a configuration
of guards. A vertex $v$ is \emph{occupied} if there is a guard on $v$,
otherwise $v$ is \emph{unoccupied}. We say that a (not necessarily dominating)
set $X$ \emph{protects} a vertex $v$, or $v$ is \emph{protected }(by $X$), if
$v$ or one of its neighbors is occupied.

Each eternal domination problem can be modeled as a two-player game,
alternating between a \emph{defender} and an \emph{attacker}: the defender
chooses $D_{1}\in\mathcal{D}_{k}$ as well as each $D_{i}$, $i>1$, while the
attacker chooses the locations $r_{1},r_{2},\ldots$ of the attacks; we say the
attacker \emph{attacks} the vertices $r_{i}$. The game starts with the
defender choosing $D_{1}$. For $i\geq1$, the attacker attacks $r_{i}\in
V-D_{i}$, and the defender \emph{defends against} the attack by choosing
$D_{i+1}\in\mathcal{D}_{k}$ subject to constraints (described below) that
depend on the particular game. The defender wins the game if they can
successfully defend the graph against any sequence of attacks, including
sequences that are infinitely long, subject to the constraints of the game;
the attacker wins otherwise.


In  the \textbf{eternal domination problem}, $D_{i}\in\mathcal{D}_{k}$ for
each $i\geq1$, $r_{i}\in V-D_{i}$, and $D_{i+1}\in\mathcal{D}_{k}$ is obtained
from $D_{i}$ by moving a guard to $r_{i}$ from an adjacent vertex $v\in D_{i}%
$. If the defender can win the game with the sets $\{D_{i}\}\subseteq
\mathcal{D}_{k}$, then each such $D_{i}$ is an \emph{eternal dominating set}.
The smallest integer $k$ such that $\mathcal{D}_{k}$ contains eternal
dominating sets is the \emph{eternal domination number} $\gamma^{\infty}(G)$.
This problem was first studied by Burger et al. in \cite{BCG2}. As an example, $\gamma(C_{5}
)=\alpha(C_{5})=2$, $\gamma^{\infty}(C_{5})=3$.

We say an attack on an unoccupied vertex $u$ is \emph{defended} by a set
$D\in\mathcal{D}_{k}$ or by (a guard on) a vertex $v\in D$ if both $D$ and
$(D-\{v\})\cup\{u\}$ belong to the collection $\{D_{i}\}\subseteq
\mathcal{D}_{k}$ of eternal dominating sets.

As first observed by Burger et al. \cite{BCG2}, it is easy to see
that $\gamma^{\infty}$ lies between the independence and clique covering numbers.

\begin{Fact}
\label{FactED_Bound}For any graph $G$, $\gamma(G) \leq\alpha(G)\leq
\gamma^{\infty}(G)\leq\theta(G)$.
\end{Fact}

The Gr\"{o}tzsch graph (which has order $11$)
is the smallest $4$-chromatic triangle-free graph, and its
complement is a graph with $\gamma^{\infty}<\theta$ and
with $\alpha<\gamma^{\infty}<\theta$, see \cite{GHH}.  Virgile cataloged the graphs with 10 and 11 vertices satisfying
$\gamma^{\infty}<\theta$ \cite{vv2}.

\section{The Gamma-Theta Conjecture}

The following conjecture stems from \cite{KM2} and is labeled the Gamma-Theta conjecture in \cite{HHHKM}.

\begin{conjecture}
[The $\gamma-\theta$ Conjecture]\label{theta_conj}For every graph $G$, if
$\gamma(G)=\gamma^{\infty}(G)$, then $\gamma(G)=\theta(G)$.
\end{conjecture}

Conjecture \ref{theta_conj} is known to be true for perfect graphs,
triangle-free graphs, graphs with maximum degree at most three \cite{KM8}, and planar graphs \cite{Ta}.

Let $H$ be the Berlekamp-van Lint-Seidel graph: a strongly regular graph with 243 vertices with each vertex having degree 22. This graph is derived from the from the Perfect Ternary Golay Code. Let $G=(V, E)$ be the complement of this graph.
It is easy to write a short computer program to verify that $\gamma(H) = 3$. See appendix (a similar program shows that there is no dominating set of size two). In fact, since $H$ is strongly regular with parameters (243, 22, 1, 2), $G$ is strongly regular with parameters
(243, 198, 199, 200). The third and fourth parameters indicate the number of common neighbors each pair of adjacent vertices have (third parameter), and
every two non-adjacent vertices have (fourth parameter). Therefore, no two vertices of $G$ can possibly dominate all 243 vertices.

 We now claim that the maximum clique size of $G$ is at most 45. 
 
The adjacency matrix has spectrum
\[
\operatorname{Spec}(G)=\{220^{(1)},4^{(110)},(-5)^{(132)}\}.
\]

$G$  is $220$-regular and its smallest eigenvalue is s=-5. Applying the Delsarte-Hoffman's bound \cite{Del} for the clique number of a regular graph:
\[
\omega(G)\le 1-\frac{k}{s},
\]
where $k$ is the degree and $s$ is the smallest eigenvalue.
Substituting gives
\[
\omega(G)
\le 1-\frac{220}{-5}
=1+44
=45.
\]

Therefore
\[
{\omega(H)\le 45}.
\]

which implies $\theta(G) \geq 5$.

\medskip

We now argue that  $\gamma^{\infty}(G) = 3$, which proves that $G$ is a counterexample to the Gamma-Theta conjecture.

\medskip

\noindent\textbf{Proof}.\hspace{0.1in}
Our  strategy is to keep three guards on a dominating triple of vertices at all times. A variation of the previous program was written that determined that $G$ has 1,987,821 dominating sets of size three. The program verified that for each such set
 $D=\{a, b, c\}$ and every possible unoccupied vertex $r \notin D$, an attack on $r$ can be defended by moving one of $\{a, b, c\}$ to $r$ (across an edge) so that the resulting configuration of guards is a dominating set of size three. Thus every dominating set 
 of size three in $G$ is an eternal dominating set and  $\gamma^{\infty}(G)=3$.
 $\Box$

We leave as an open problem to find smaller graphs with this property. We note our suspicions that such graphs would likely have a lot of `symmetry' like $G$ or like the complements of Kneser graphs used in \cite{GK} used to proved bounds on  $\gamma^{\infty}$.

\section{Contribution Statement}

The counterexample graph was found by GPT-6 Astra in an experiment run by the first author,
Tom Adamczewski, whose contribution to this paper was limited to setting up and conducting this
experiment.

\newpage

\section{Appendix}

\begin{verbatim}
import csv
import itertools

A = []
with open("b.csv", newline="") as f:
    reader = csv.reader(f)
    for row in reader:
        A.append([int(x) for x in row])

n = len(A)

closed_neighborhoods = []

for v in range(n):
    N = {v}
    for u in range(n):
        if A[v][u] != 0:
            N.add(u)
    closed_neighborhoods.append(N)

V = set(range(n))

found = False
for S in itertools.combinations(range(n), 3):

    dominated = (
        closed_neighborhoods[S[0]]
        | closed_neighborhoods[S[1]]
        | closed_neighborhoods[S[2]]
    )
    if dominated == V:
        print("Dominating set of size 3 found:")
        print(S)
        found = True
        break
\end{verbatim}

\end{document}